\documentclass[9pt]{amsart}

\usepackage{booktabs}
\usepackage{amsmath,mathtools,amsthm,amsfonts,amssymb,color,verbatim,graphicx}

\usepackage[letterpaper]{geometry}
\usepackage{stmaryrd,tikz}

\newtheorem{theorem}[equation]{Theorem}
\newtheorem{corollary}[equation]{Corollary}

\newtheorem{lemma}[equation]{Lemma}
\newtheorem{proposition}[equation]{Proposition}

\theoremstyle{definition}

\numberwithin{equation}{section}
\newcommand{\ZZ}{\mathbb{Z}}

\begin{document}
\setlength{\jot}{0pt} % this is to remove extra space between lines in aligned environment.
\title{Group bases for some solvable groups and semidirect products}

\today
\author{Bret Benesh and Jason Lutz}

\begin{abstract}
A set $B$ is a basis for a vector space $V$ if every element of $V$ can be uniquely written as a linear combination of the elements of $B$.  There is a similar definition of a basis for a finite group.  We show that certain semidirect products of finite groups---including all semidirect products of finite abelian groups---have bases; any group of order $m$ or $2m$ for odd, cube-free $m$ has a basis; and the quaternions do not have a basis.  
\end{abstract}

\maketitle

\begin{section}{Introduction and Preliminaries}
Let $G$ be a finite group, and suppose there exists $(b_1,\ldots,b_n) \in (G \setminus \{e\})^n$ such that for every $g \in G$, there are unique $a_i \in \{0,\ldots,|b_i|-1\}$ such that $g=b_1^{a_1}\cdots b_n^{a_n}$.  In this case, we say that $(b_1,\ldots,b_n)$ is a \emph{basis} for $G$ and $(a_1,\ldots,a_n)$ is a \emph{representation} for $g$, and it is easy to see that $|b_1| \cdots |b_n|=|G|$.  If $G=\{e\}$, we say that $G$ has a \emph{trivial basis}.  

Philip Hall~\cite{hall1934contribution} introduced the idea of a group basis, and Charles Hopkins~\cite{hopkins1937concerning} proved that dihedral, symmetric, and alternating groups all have bases.   There are several other families of groups that have natural bases.  If $G=\langle b \rangle$ is cyclic, then $(b)$ is easily seen to be a basis of $G$.  The Fundamental Theorem of Abelian Groups states that if $A$ is a finite abelian group, then there are $b_i \in A$ and prime powers $q_i$ such that $A \cong \langle b_1 \rangle \times \ldots \langle b_m \rangle$, where $|b_i|=q_i$.  It is easy then to see that $(b_1,\ldots,b_m)$ is a basis for $A$.  Finally, groups can have more than one basis:  $\left((1,1)\right)$ and $\left((1,0),(0,1)\right)$ are both bases for $\ZZ_2 \times \ZZ_3$.

%(JML) The notation in the preceding sentence doesn't match the next
%one. One way to address this discrepancy would be to replace $(b_1,\ldots,b_m)$
%with $((b_1,0,\ldots 0), \ldots, (0,\ldots, 0,b_m))$. This is
%cumbersome, and may therefore be an unnecessary change.
%(BJB):  I defined the b_i to be elements of A (which I should have done before) to take care of this.

A group basis is a special case of the following definition by Magliveras~\cite{PGMIntro}: Let $G$ be a finite group, and let $\alpha$ be a sequence $(\alpha_1,\ldots,\alpha_n)$, where each $\alpha_i$ is itself a sequence $(a_{i,1},\ldots,a_{i,r_i})$ of elements of $G$.  Then $\alpha$ is said to be a \emph{logarithmic signature} for $G$ if $|G|=r_1  \cdots r_n$ and each $g \in G$ is uniquely represented as a product $g=a_{1,j_i} \cdots a_{n,j_n}$ with $a_{i,j_i} \in \alpha_i$.  Logarithmic signatures have been considered for cryptographic systems in~\cite{MST3Intro},~\cite{MagliverasNewApproaches}, and~\cite{MST3Analysis}.  Then a basis for a finite group $G$ is a logarithmic signature $(\alpha_1,\ldots,\alpha_n)$ such that each $\alpha_i$ is a cyclic subgroup of $G$. All solvable groups have a logarithmic signature~\cite{vasco2002obstacles}.  For instance, recall that the quaternions $Q_8$ is the set $\{\pm 1, \pm i, \pm j, \pm k\}$, where $|1|=1$, $|-1|=2$ and $|\pm x|=4$ for $x \in \{i,j,k\}$.     One such logarithmic signature for $Q_8$ is $\left((1,i,-1,-i), (1,j)\right)$.  However, this logarithmic signature is not a basis because $\{1,j\}$ is not equal to a cyclic subgroup of $Q_8$. In fact, the proposition below demonstrates that $Q_8$ has no basis, so we say that that $Q_8$ is \emph{basisless}.
%(JML) More is required: Distinct cyclic groups $\alpha_i$ and $\alpha_j$
%must have trivial intersection, else (at
%least) one element will have a non-unique representation.
%(BJB)  I don't think so---elements of G must be uniquely represented
%by the logarithmic signature, which automatically implies that the
%cyclic subgroups are disjoint.
%(JML)  You are correct, and I withdraw my previous concern.

\begin{proposition}\label{prop:quaternions}
The quaternions $Q_8$ do not have a basis.
\end{proposition}
\begin{proof}
Suppose $Q_8$ has a basis $(b_1,\ldots,b_n)$.  Then $|b_1| \cdots |b_n|=8$, and we see that $|b_i| \in \{2,4,8\}$ with $ 1 \leq n \leq 3$. Since $Q_8$ is noncyclic, we can refine this to $|b_i| \in \{2,4\}$ and $n \in \{2,3\}$.  Since $Q_8$ has a unique element $-1$ of order $2$, we conclude that the basis is either $(-1,b)$ or $(b,-1)$ for some $b \in Q_8$ with $|b|=4$.  Then $b^2$ has order $2$, so $b^2=-1$.  If the basis is $(-1,b)$, we then have $(0,2)$ and $(1,0)$ are both representations for $-1$, which contradicts the uniqueness of representations of the basis.  Similarly, $(0,1)$ and $(2,0)$ are representations for $-1$ for the basis $(b,-1)$, which is also a contradiction.  Therefore, no basis can exist for $Q_8$. 
\end{proof}

Groups involving $Q_8$ can still have a basis.  For instance, the semidihedral group of order $16$ with presentation $\langle a,x \mid a^8=x^2=e, x^{-1}ax=a^{-1}\rangle$ has a normal subgroup $\langle a^2,ax \rangle$ isomorphic to the quaternions, yet has a basis $(a,x)$.  Also, there is a group $G$ of order $16$ with presentation $\langle a,x \mid a^4=x^4=e,x^{-1}ax=a^{-1}\rangle$ such that $G/\langle a^2x^2\rangle \cong Q_8$, yet $G$ has a basis $(a,x)$ (Corollary~\ref{cor:dihedral} also demonstrates that $G$ has a basis, since $G$ is isomorphic to the semidirect product $\mathbb{Z}_4 \rtimes \mathbb{Z}_4$ with an inversion action).  Finally, the symmetric group $S_8$ has a non-normal subgroup isomorphic to the quaternions, and $S_8$ has a basis~\cite{hopkins1937concerning}.  So it is possible for a group with a basisless subgroup (normal or otherwise) or basisless quotient to have a basis.  

%(JML) I feel that ``Additionally'' is not a fitting way to start this
%paragraph, and that ``However'' should be used. Since this paragraph
%establishes (among other things) that 
%a good structure (having a basis) can occur despite having a nasty
%subgroup (one isomorphic to $Q_8$), this paragraph contrasts with the
%previous paragraph. 
%(BJB):  Agreed.

We end this section with an important lemma.

\begin{lemma}\label{lem:SubgroupCoset}
Let $G$ be a finite group, and let $H$ be a subgroup of $G$ with a basis $(h_1,\ldots,h_m)$.  Suppose there exists $b_1,\ldots,b_n \in G\setminus H$ such that, for all $x=b_1^{c_1}\cdots b_n^{c_n}$ and $y=b_1^{k_1}\cdots b_n^{k_n}$ with $c_i,k_i \in \{0,\ldots,|b_i|-1\}$, the following hold: 

\begin{enumerate}
\item\label{item:powers} $x=y$ if and only if $(c_1,\ldots,c_n)=(k_1,\ldots,k_n)$,
\item\label{item:coset} $xy^{-1} \in H$ if and only if $x=y$, and 
\item\label{item:size} $|G:H| = |b_1|\cdots |b_n|$.
\end{enumerate}

\noindent
Then $B=(h_1,\ldots,h_m,b_1,\ldots,b_n)$ is a basis for $G$.
\end{lemma}
\begin{proof}
%By Item~(\ref{item:coset}), we have $Hx$ and $Hy$ are distinct for distinct $x$ and $y$.  By Item~(\ref{item:powers}), there are $|b_1|\cdots |b_n|$ distinct cosets of $H$.  By Item~(\ref{item:size}), every coset of $H$ is achieved by $Hb_1^{c_1}\cdots b_n^{c_n}$ for some $(c_1,\ldots,c_n)$. 

We start by showing that every coset of $H$ can be written as $Hb_1^{c_1}\cdots b_n^{c_n}$ for some $(c_1,\ldots,c_n) \in \prod_{i=1}^n\{0,\ldots,|b_i|-1\}$.  Let $\psi$ be a map from $\prod_{i=1}^n\{0,\ldots,|b_i|-1\}$ to the set of right cosets of $H$ defined by $\psi(c_1,\ldots,c_n)=Hb_1^{c_1}\cdots b_n^{c_n}$.  Then $\psi$ is injective by Items~(\ref{item:powers}) and~(\ref{item:coset}) and surjective by Item~(\ref{item:size}), giving the result.  

%(JML) With the placement of the page break, this (initially) reads as a
%statement that the entire proof is complete. With that in mind, I
%suggest the following minor change: Indeed, let $\psi$ be a map from
%$\prod_{i=1}^n\{0,\ldots,|b_i|-1\}$ to the set of right cosets of $H$
%defined by $\psi(c_1,\ldots,c_n)=Hb_1^{c_1}\cdots b_n^{c_n}$; 
%$\psi$ is injective by Items~(\ref{item:powers})
%and~(\ref{item:coset}) and surjective by Item~(\ref{item:size}).
%(BJB):  I would prefer not to do this, since the font
%size/formatting/etc is all going to be changed by the publisher
%anyway.  I would rather make the changes then rather than now, since
%the page breaks could be completely different (let me know if you
%think your suggestion is simply superior to what I have, pagebreak or
%not). 
%(JML)  I agree with your preference to not make such a formatting
%decision at this time. I withdraw my suggestion.
 
Let $g \in G$. Then $g$ appears in exactly one coset of $H$, so let $(c_1,\ldots,c_n)$ be such that $g \in Hb_1^{c_1}\cdots b_n^{c_n}$.  Then $g=hb_1^{c_1}\cdots b_n^{c_n}$ for some $h \in H$, and we can write $h$ uniquely as $h_1^{l_1}\cdots h_m^{l_m}$ for some $(l_1,\ldots,l_m)$.  Then \[g=h_1^{l_1}\cdots h_m^{l_m} \cdot b_1^{c_1} \cdots b_n^{c_n}\] for unique $l_i \in \{0,\ldots,|h_i|-1\}$ and unique $c_i \in \{0,\ldots,|b_i|-1\}$.  Therefore, $(h_1,\ldots,h_m,b_1,\ldots,b_n)$ is a basis for $G$.  
\end{proof}
\end{section}

\begin{section}{Families of groups}
\begin{subsection}{Semidirect products of abelian groups}
Let $H$ and $N$ be finite groups such that there is a map $\phi: H \to \operatorname{Aut}(N)$, and let $\phi_{h}:N \to N$ denote $\phi(h)$. Recall that a \emph{semidirect product $N \rtimes_{\phi} H$} is defined to be the set $\{(h,n) \mid h \in H, n \in N\}$ together with the operation defined by $(h_1,n_1)(h_2,n_2)=(h_1h_2,\phi_{h_2}(n_1)n_2)$ for all $h_i \in H$ and $n_i \in N$.  

The following results generalizes the results about dihedral groups from~\cite{hopkins1937concerning}.

\begin{theorem} 
Let $H$ and $K$ be finite groups with bases such that there is a homomorphism $\phi:H \to \operatorname{Aut}(K)$.  If $G=K \rtimes_{\phi} H$, then $G$ has a basis.  
\end{theorem} 
\begin{proof}
Let $(h_1,\ldots,h_n)$ be a basis for $H$ and  $(k_1,\ldots,k_m)$ be a basis for $K$.   Then $h_1^{a_1}\cdots h_n^{a_n}=h_1^{c_1}\cdots h_n^{c_n}$ for some $a_i,c_i \in \{0,\ldots,|h_i|-1\}$ if and only if $(a_1,\ldots,a_n)=(c_1,\ldots,c_n)$ since $(h_1,\ldots,h_n)$ is a basis for $H$, $(h_1^{a_1}\cdots h_n^{a_n})(h_1^{c_1}\cdots h_n^{c_n})^{-1}\in K$ if and only if $h_1^{a_1}\cdots h_n^{a_n}=h_1^{c_1}\cdots h_n^{c_n}$ since $H \cap K = \{(e,e)\}$, and $|G:K| = |H| =  |h_1|\cdots |h_n|$.  Therefore, the $(n+m)$-tuple $(k_1,\ldots,k_m,h_1,\ldots,h_n)$ is a basis for $G$ by Lemma~\ref{lem:SubgroupCoset}.
\end{proof}

The following corollary holds because $A$ and $B$ have bases as a consequence of the Fundamental Theorem of Abelian Groups.  In particular, every dihedral group has a basis.

\begin{corollary}\label{cor:dihedral}
Every finite group isomorphic to $A \rtimes_{\phi} B$ for abelian groups $A$ and $B$ and some map $\phi: B \to \operatorname{Aut}(A)$ has a basis.
\end{corollary}
\end{subsection}

\begin{subsection}{Special cases of solvable groups}

Proposition~\ref{prop:quaternions} demonstrates that not every solvable group has a basis.  Still, we can demonstrate some special cases when solvable groups have bases. Recall that a subgroup $H$ of a finite group $G$ is a \emph{Hall $\pi$-subgroup} for a set of primes $\pi$ if $|H|$ and $|G:H|$ are coprime and every prime divisor of $|H|$ is in $\pi$, and a result of Philip Hall~\cite[Theorem~8.9]{Isaacs1994} states that every solvable $G$ has a Hall $\pi$-subgroup for any set of primes $\pi$.

\begin{theorem}\label{thm:Solvable}
Let $G$ be a finite solvable group.  If every Sylow subgroup of $G$ has a basis, then $G$ has a basis.
\end{theorem}
\begin{proof}
Let $\pi(G)$ denote the primes that divide $|G|$.  We will proceed by induction.  If $|\pi(G)|=1$, then $G$ is equal to its sole Sylow subgroup, and thus has a basis by assumption.

Now suppose $|\pi(G)| > 1$.  Let $p \in \pi(G)$ and $\pi'(G) = \pi(G)-\{p\}$.  Because $G$ is solvable, $G$ has a proper Hall $\pi'(G)$-subgroup $H$.  By induction, $H$ has a basis $(h_1,\ldots,h_m)$.  Let $P$ be a Sylow $p$-subgroup of $G$; then $P$ has a basis $(b_1,\ldots,b_n)$ by assumption.  It is easy to check that $H$ and $(b_1,\ldots,b_n)$ fulfill the hypotheses of Lemma~\ref{lem:SubgroupCoset}, so $G$ has a basis.
\end{proof}

Recall that an integer $m$ is \emph{cube-free} if $p^3$ fails to divide $m$ for every prime $p$.  

\begin{corollary}  
Let $m$ be an odd integer that is cube-free.  If $G$ has order $m$ or $2m$, then $G$ has a basis.
\end{corollary}  
\begin{proof}
Let $p$ be a prime number that divides $|G|$, and let $P$ be a Sylow $p$-subgroup of $G$.  Then $P$ must have order at most $p^2$ because $|G|$ is cube-free, which implies that $P$ is abelian and thus has a basis.  If $|G|=m$,  then $G$ is solvable by the Odd Order Theorem~\cite{FeitThompson}.  If $|G|=2m$,  then $G$ is solvable by the Odd Order Theorem,~\cite[Corollary~6.12]{Isaacs1994}, and~\cite[Corollary~8.4]{Isaacs1994}.  Theorem~\ref{thm:Solvable} then implies $G$ has a basis.
\end{proof}

\end{subsection}

\begin{section}{Open Questions}

We close with several open questions about group bases.

\begin{enumerate}
\item We have seen that $Q_8$ does not have a basis.   When exactly does a nilpotent group have a basis?   
\item Similarly, when does a solvable group have a basis?   When does a non-solvable group have a basis?
%\item If $H$ is a group that does not have a basis and $G$ has a normal subgroup isomorphic to $H$, can $G$ have a basis?  YES:  Semidihedral of order 16 
%\item If $H$ is a group that does not have a basis and $G$ has a quotient isomorphic to $H$, can $G$ have a basis?   YES:  C4:C4 
\end{enumerate}
\end{section}

%(JML) General comment: It is part of the definition of basis that the
%group must be finite, and in several places (for example, the
%statement of Theorem 2.1), the hypothesis is that a group has a basis
%and that the group is finite. Must we state both, or can we remove
%some of the references to the group being finite?
%(BJB):  I think that this could go either way.  Some of the finites
%are technically needed, like Corollary 2.2:  there are infinite
%groups isomorphic to A:B, but these infinite groups could not have a
%basis. One could argue (and I have another co-author who often does)
%that the mention of the word "basis" implicitly says that G=A:B has
%to be finite.  I tend to prefer to be more explicit about these
%things, but it is often a matter of preference.  What is your
%opinion? 
%(BJB):  But we definitely don't need the word 'finite' in Lemma 1.2
%or Theorem 2.1 (I am not sure if we need it for Theorem 2.3, since I
%think that one can define something like a Sylow subgroup (not based
%on order, of course) for an infinite group if one is careful).
%(JML)  I am happy to defer to your preference to be explicit in these
%things.

\end{section}
\bibliographystyle{amsplain}
\bibliography{CryptoBibliography} 

\end{document}